\documentclass[10pt,reqno, english]{amsart}  
\usepackage[utf8]{inputenc}
\usepackage[T1]{fontenc}
\usepackage{amsmath,amsthm}
\usepackage{amsfonts,amssymb}
\usepackage{url}
\usepackage{mathtools}  
\usepackage[colorlinks=true,urlcolor=blue,linkcolor=red,citecolor=magenta]{hyperref}
\usepackage{paralist}
\usepackage{anysize}

\theoremstyle{plain}
\newtheorem{theorem}{Theorem}[section]
\newtheorem{lemma}[theorem]{Lemma}

\newtheorem{corollary}[theorem]{Corollary}

\theoremstyle{definition}

\newtheorem{example}[theorem]{Example}

\newtheorem{question}[theorem]{Question}

\newcommand{\R}{\mathbb{R}}
\newcommand{\Z}{\mathbb{Z}}

\begin{document}

\title[Splitting loops and necklaces]{Splitting loops and necklaces: \\ Variants of the square peg problem}



\author[J. Aslam, S. Chen, F. Frick, S. Saloff-Coste, L. Setiabrata, H. Thomas]{Jai Aslam, Shujian Chen, Florian Frick, Sam Saloff-Coste, \\ Linus Setiabrata, and Hugh Thomas}

\address[JA, SC]{Dept.\ Math., Northeastern University, Boston, MA 02115, USA}
\email{\{aslam.j, chen.shuj\}@husky.neu.edu}

\address[FF, SSC, LS]{Dept.\ Math., Cornell University, Ithaca, NY 14853, USA}
\email{\{ff238, sps247, ls823\}@cornell.edu} 

\address[HT]{Math.\ Dept., Universit\'e du Qu\'ebec \'a Montr\'eal, Canada}
\email{thomas.hugh\_r@uqam.ca}

\date{\today}
\maketitle


\begin{abstract}
\small
Toeplitz conjectured that any simple planar loop inscribes a square. Here we prove variants of Toeplitz' square peg problem.
We prove Hadwiger's 1971 conjecture that any simple loop in $3$-space inscribes a parallelogram. We show that any simple 
planar loop inscribes sufficiently many rectangles that their vertices are dense in the loop (independently due to Schwartz). 
If the loop is rectifiable, there is a rectangle that cuts the loop into four pieces that can be rearranged to
form two loops of equal length. A rectifiable loop in $d$-space can be cut into $(r-1)(d+1)+1$ pieces that can be rearranged 
by translations to form $r$ loops of equal length. We relate our results to fair divisions of necklaces in the sense of Alon and 
to Tverberg-type results. This provides a new approach and a common framework to obtain variants of Toeplitz' square peg 
problem for the class of all continuous curves.
\end{abstract}

\section{Introduction}

Toeplitz~\cite{toeplitz1911} conjectured that an embedded continuous closed curve (a \emph{loop}) in the plane \emph{inscribes} a square,
that is, it contains the four vertices of a square. This conjecture has been settled in several special cases,
such as piecewise analytic curves (Emch~\cite{emch1916}), $C^2$ curves (Schnirelman~\cite{schnirelman1944}, see also 
Guggenheimer~\cite{guggenheimer1965}), $C^1$ curves (Stromquist~\cite{stromquist1989}), or homotopically nontrivial loops 
contained in certain annuli, and an open and dense class of curves (Matschke~\cite{matschke2011}); also see Matschke's survey~\cite{matschke2014}.
Recently, Tao~\cite{tao2017} provided a novel approach to Toeplitz' conjecture proving it for curves that arise as the union of two 
graphs of Lipschitz functions with Lipschitz constant less than one. Results for the class of all continuous closed curves are rare. 
It seems that the most general statements towards Toeplitz' conjecture are that any loop inscribes a rhombus with two sides
parallel to a given line (see Nielsen~\cite{nielsen1995}) and that any loop inscribes a rectangle; this was proven by Vaughan, and the 
proof appears in Meyerson's manuscript~\cite{meyerson1981}. See also Pak's book~\cite[Prop. 5.4]{pak2010} and Schwartz' recent 
trichotomy of inscribed rectangles~\cite{schwartz2018}. For additional very recent progress on special inscribed quadrilaterals 
see~\cite{akopyan2017, hugelmeyer2018, matschke2018}.

Nielsen's result proceeds by approximating continuous curves by piecewise linear curves while certifying that the rhombus
does not degenerate in this process. Similarly, Schwartz approximates loops by generic polygons. 
Vaughan's result is particular to the case of inscribed rectangles and does not lend
itself easily to proving variants. Here we describe a novel technique that proves relatives of Toeplitz' conjecture for all continuous
curves in the same generalized fashion without a need for approximation. 

An important variant of the square peg problem is a 1971 conjecture of Hadwiger~\cite{hadwiger1971} that states that any loop in $\R^3$
inscribes a parallelogram. 
Guggenheimer~\cite{guggenheimer1974} established this for $C^2$ curves and Makeev~\cite{makeev2005}
for $C^1$ curves. Vre\'cica and \v Zivaljevi\'c~\cite{vrecica2011-2} develop a general proof method that also yields Hadwiger's conjecture
for $C^1$ curves. In fact, all of these results establish the existence of an inscribed rhombus.

We relate inscribing special $n$-gons into loops to results of fair division on the real line, such as the Hobby--Rice theorem in
$L^1$ approximation (see Theorem~\ref{thm:hobby-rice}) as well as its generalizations. We prove the following results:
\begin{compactitem}
	\item Hadwiger's conjecture holds: any simple loop in~$\R^3$ inscribes a parallelogram. In fact, it inscribes so many parallelograms
		that the set of vertices is dense in the loop; see Theorem~\ref{thm:hadwiger}. Here we allow parallelograms that consist of four
		pairwise distinct points on a line and that are the limit of a sequence of parallelograms (so does Hadwiger).
	\item Any simple planar loop inscribes sufficiently many rectangles that the set of vertices is dense in the loop; see Theorem~\ref{thm:many-rectangles}.
		Schwartz~\cite{schwartz2018} recently and independently proved that all but at most four points of a loop are the vertices of a rectangle.
	\item Any rectifiable simple planar loop inscribes a rectangle that cuts the loop into four parts $\gamma^{(1)}$, $\gamma^{(2)}$, 
		$\gamma^{(3)}$, $\gamma^{(4)}$ in
		cyclic order such that the total length of $\gamma^{(1)}$ and $\gamma^{(3)}$ is equal to the total length of $\gamma^{(2)}$ and~$\gamma^{(4)}$; 
		see Theorem~\ref{thm:inscribed}.
	\item Any rectifiable loop in $\R^d$ can be cut into $(r-1)(d+1)+1$ pieces that may be rearranged by translations to form $r$ loops of equal length;
		see Theorem~\ref{thm:splitting-loops}.
	\item We prove a proper extension of Alon's necklace splitting result~\cite{alon1987} for divisions of the unit interval into a prime number of parts
		by applying the topological machinery of the optimal colored Tverberg theorem of Blagojevi\'c, Matschke, and Ziegler~\cite{blagojevic2015}; 
		see Theorem~\ref{thm:opt-splitting-necklaces}. This allows us to prove a proper strengthening of Theorem~\ref{thm:splitting-loops}
		for primes $r$; see Corollary~\ref{cor:opt-splitting-loops}.
\end{compactitem}

\section{Inscribing parallelograms and rectangles}
\label{sec:rectangles}

We find it instructive to first discuss why any planar $C^1$ loop inscribes a parallelogram with a prescribed vertex. 
This result follows easily from the Hobby--Rice theorem below. After deducing this special case, we will discuss how 
to obtain generalizations.

\begin{theorem}[Hobby and Rice~\cite{hobby1965}]
\label{thm:hobby-rice}
	Let $\mu$ be a finite nonatomic real measure on~$[0,1]$. Let ${f_i \colon [0,1] \longrightarrow \R}$, $ i =1, \dots, n$,
	be functions in~$L^1(d\mu)$. Then there are points $t_i$ with $0 = t_0 \le t_1 \le \dots \le t_n \le t_{n+1} = 1$
	such that
	$$\sum_{j=1}^{n+1} (-1)^j \int_{t_{j-1}}^{t_j} f_i(t) \ d\mu(t) = 0 \quad \text{for every} \ i =1, \dots, n.$$
\end{theorem}

Let $\gamma\colon [0,1] \longrightarrow \R^2, t \mapsto (\gamma_1(t), \gamma_2(t))$ be a $C^1$ loop in the plane.
We note that $$\int_0^1 \gamma'_i(t) \ dt = \gamma_i(1) - \gamma_i(0) = 0 \quad \text{for} \ i =1,2.$$
The Hobby--Rice theorem implies that there are three points $0 \le a \le b \le c \le 1$ such that 
\begin{equation}
\label{eq:der}
	\int_0^a |\gamma'(t)| \ dt + \int_b^c |\gamma'(t)| \ dt = \int_a^b |\gamma'(t)| \ dt + \int_c^1 |\gamma'(t)| \ dt
\end{equation}
and
$$\int_0^a \gamma'_i(t) \ dt + \int_b^c \gamma'_i(t) \ dt = \int_a^b \gamma'_i(t) \ dt + \int_c^1 \gamma'_i(t) \ dt,$$
which implies that both sides of this latter equation vanish. This implies that $\gamma(a) - \gamma(0) = \gamma(b) - \gamma(c)$
and $\gamma(b) - \gamma(a) = \gamma(c) - \gamma(1)$. This implies that the points $\gamma(0), \gamma(a), \gamma(b)$, 
and $\gamma(c)$ describe a parallelogram inscribed into~$\gamma$, where the vertex $\gamma(0)$ was prescribed in advance.
Equation~(\ref{eq:der}) ensures that the parallelogram is non-degenerate.

The requirement that $\gamma$ be continuously differentiable may be relaxed to $\gamma$ being continuous since we 
differentiate $\gamma$ and then integrate again. This will require a slight extension of the Hobby--Rice theorem. In fact,
we will immediately prove a version that instead of splitting $[0,1]$ into positive and negative subintervals, splits a partition
of $[0,1]$ into $r$ parts that equalize given functions on the intervals of each part. One such extension of the Hobby--Rice theorem is due
to Alon~\cite{alon1987}. The theorem below is a slight modification, but can be proven in a similar way. We also refer to
the statement and proof in Matou\v sek's book~\cite{matousek2008}.

\begin{theorem}
\label{thm:splitting-necklaces}
	Let $f_1, \dots, f_m\colon [0,1] \longrightarrow \R$ be continuous functions. Let $r \ge 2$ be an integer, and set~${n = (r-1)m}$.
	Then there are points $0 = t_0 \le t_1 \le \dots \le t_{n+1} = 1$ and a partition of the set $[n+1]$ into subsets $T_1, \dots, T_r$ such that
	$$\sum_{j\in T_1} f_k(t_j) - f_k(t_{j-1}) = \sum_{j\in T_2} f_k(t_j) - f_k(t_{j-1}) =  \dots = \sum_{j\in T_r} f_k(t_j) - f_k(t_{j-1}), \quad k = 1, \dots, m.$$
\end{theorem}

Alon's theorem guarantees a fair splitting of measures $\mu_1, \dots, \mu_m$ on $[0,1]$ that are continuous in the sense that $\int_0^x \ d\mu_k$ is 
continuous in~$x$. We recover this case by setting $f_k(x) = \int_0^x \ d\mu_k$. The popular interpretation of Alon's theorem is that $r$
thieves have stolen a necklace with $m$ kinds of beads, whose densities along the necklace are given by $\mu_1, \dots, \mu_m$. Then the 
thieves can split the necklace with $(r-1)m$ cuts such that each thief receives an equal amount of each kind of bead.

We first need some notation before we can prove this result. By $W_r = \{(y_1, \dots, y_r) \in \R^r \ | \ \sum y_i = 0\}$
we denote the standard representation of the symmetric group~$S_r$. For abstract simplicial complexes $K$ and $L$ on disjoint
vertex sets denote their join by~${K * L}$, that is, the abstract simplicial complex whose faces are $\sigma \cup \tau$ with $\sigma \in K$
and $\tau \in L$. If we take the join of simplicial complexes whose vertex set is not disjoint to begin with, such as $K * K$, we first force 
the vertex sets to be disjoint. The $r$-fold, deleted join of $K$, denoted $K^{*r}_\Delta$, is the subcomplex of the $r$-fold join of $K$,
where unions of faces $\sigma_1, \dots, \sigma_r$ that were not pairwise disjoint to begin with have been deleted. We refer to
Matou\v sek~\cite{matousek2008} for details. 
Given two topological spaces $X$ and $Y$ with $G$-actions, we call a continuous map $f\colon X
\longrightarrow Y$ equivariant (or $G$-equivariant) if $f(g\cdot x) = g \cdot f(x)$ for all $x \in X$ and~${g \in G}$.

Matou\v sek~\cite[Theorem 6.6.1]{matousek2008} describes how points in the $r$-fold deleted join $(\Delta_n)^{*r}_\Delta$ of the $n$-simplex~$\Delta_n$
correspond to $n$ points $0 \le t_1 \le \dots \le t_n \le 1$ and partitions of $[n+1]$ into $r$ parts. We describe an alternative way of seeing this 
parametrization in the proof below. It follows from a theorem of Dold~\cite{dold1983} that for $n = (r-1)m$ and $r$ a prime,
any $S_r$-equivariant map $(\Delta_n)^{*r}_\Delta \longrightarrow W_r^{\oplus m}$ must include the origin in its image; 
see~\cite[Corollary 6.4.4]{matousek2008}.

\begin{proof}[Proof of Theorem~\ref{thm:splitting-necklaces}]
	First let $r \ge 2$ be a prime. We will induct on the number of prime divisors in the end.
	We first describe how points in the $r$-fold deleted join $(\Delta_n)^{*r}_\Delta$ of an $n$-simplex
	correspond to divisions of $[0,1]$ into $n+1$ (possibly empty) intervals, and a partition of those intervals into $r$ (possibly
	empty) parts. In the following we will identify the vertex set of $\Delta_n$ with~$[n+1]$.
	The simplicial complex $(\Delta_n)^{*r}_\Delta$ consists of joins $\sigma_1 * \dots * \sigma_r$ of $r$ pairwise
	disjoint faces $\sigma_i$ of the $n$-simplex~$\Delta_n$. A point in the geometric realization of $\sigma_1 * \dots * \sigma_r$ 
	corresponds to a convex combination $\lambda_1x_1 + \dots + \lambda_rx_r$ of points $x_i \in \sigma_i$. In particular, 
	$\lambda_i \ge 0$ and $\sum \lambda_i = 1$. 
	
	Let $\lambda_1x_1 + \dots + \lambda_rx_r$ be an arbitrary point in~$(\Delta_n)^{*r}_\Delta$. We can think of the expression
	$\lambda_1x_1 + \dots + \lambda_rx_r$ as a convex combination of points $x_i$ in the simplex~$\Delta_n$, and thus as 
	a point $x$ in the standard $n$-simplex $\Delta_n = \{(x^0, \dots, x^n) \in \R^{n+1} \ | \ x^i \ge 0 \ \text{and} \ \sum x^i = 1\}$.
	Such a point corresponds to a partition of $[0,1]$ into the $n+1$ intervals $[0,x^0],[x^0,x^0+x^1],\dots,[x^0+\dots+x^{n-1}, 1]$.
	Let $t_j$ denote $x^0+x^1+\dots+x^{j-1}$ for $j \in [n]$, $t_0 = 0$, and $t_{n+1} = 1$.
	The point $\lambda_1x_1 + \dots + \lambda_rx_r$ is in a join of pairwise disjoint faces $\sigma_1 * \dots * \sigma_r$, where
	$\sigma_i$ is the minimal supporting face of~$x_i$. To split the $n+1$ intervals into $r$ groups of intervals, let $j \in [n+1]$
	be in $T_i$ if and only if the $j$th vertex of $\Delta_n$ is contained in~$\sigma_i$ and~$\lambda_i > 0$.
	Notice that if $j$ is not contained in any $T_i$, then $t_j = t_{j-1}$ and we can add it to an arbitrary set~$T_i$.
	
	For each $i \in \{1, \dots, r\}$ define the continuous map
	$$F_i\colon (\Delta_n)^{*r}_\Delta \longrightarrow \R^m, 
	\lambda_1x_1 + \dots + \lambda_rx_r \mapsto \left(\sum_{j\in T_i} f_1(t_j)-f_1(t_{j-1}),  \dots, \sum_{j\in T_i} f_m(t_j)-f_m(t_{j-1})\right),$$
	and define $F\colon (\Delta_n)^{*r}_\Delta \longrightarrow (\R^m)^r$ by $F(x) = (F_1(x), \dots, F_r(x))$. There is an
	action by the symmetric group $S_r$ on $(\Delta_n)^{*r}_\Delta$ that permutes copies of~$\Delta_n$, and the map $F$
	is equivariant with respect to this action, where $S_r$ permutes the $F_i$ accordingly.
	
	Observe that if the theorem was false, then the image of $F$ would not map to the diagonal 
	$D = \{(y_1, \dots,y_r) \in (\R^m)^r \ | \ y_1 = \dots = y_r\}$.  Orthogonally projecting along the diagonal gives
	an equivariant map $\widehat F\colon (\Delta_n)^{*r}_\Delta \longrightarrow W_r^{\oplus m}$ that
	does not include the origin in its image. This is a contradiction to~\cite[Corollary 6.4.4]{matousek2008}.
	
	It remains to be shown that if the statement of the theorem holds for $r = q$ and $r = p$ then it also holds for their product
	$r = pq$. Let $[a_i,b_i] \subset [0,1]$, $i \in [\ell]$, be a collection of pairwise disjoint intervals. Denote their union 
	by~$I = \bigcup_i [a_i,b_i]$. Let $f_1, \dots, f_m \colon I \longrightarrow \R$ be continuous functions with $f_k(b_i) =f_k(a_{i+1})$
	for all $i \in [\ell-1]$ and all $k \in [m]$. Then the theorem holds in the same way for the functions $f_i$, since we can simply
	reparametrize to obtain continuous functions on all of~$[0,1]$. 
	
	Assume that we have shown the theorem for $r =p$ and $r =q$. Now given continuous maps $f_1, \dots, f_m \colon [0,1] \longrightarrow \R$,
	let $n = (p-1)m$. Find points $0 = t_0 \le t_1 \le \dots \le t_{n+1} = 1$ and a partition of the set $[n+1]$ into subsets $T_1, \dots, T_p$ such that
	$$\sum_{j\in T_1} f_k(t_j) - f_k(t_{j-1}) = \sum_{j\in T_2} f_k(t_j) - f_k(t_{j-1}) =  \dots = \sum_{j\in T_p} f_k(t_j) - f_k(t_{j-1}), \quad k = 1, \dots, m.$$
	
	The sum $\sum_{i=1}^p \sum_{j\in T_i} f_k(t_j) - f_k(t_{j-1})$ telescopes and is equal to $f_k(1) -f_k(0)$. Thus 
	$\sum_{j\in T_i} f_k(t_j) - f_k(t_{j-1}) = \frac{1}{p}(f_k(1) -f_k(0))$ for all $i$ and~$k$.
	Fix one set~$T_i$ and consider $I = \bigcup_{j \in T_i} \ [t_{j-1},t_j]$. Let $y$ be the left-most point in~$T_i$, and let $z$ be
	the right-most point in~$T_i$. Define $h_k \colon I \longrightarrow \R$ by 
	$h_k(x) = f_k(x) - f_k(t_{j-1}) + \sum f_k(t_s) - f_k(t_{s-1})$ if $x \in [t_{j-1},t_j]$, where the sum is taken over all $s \in T_i$ with $s < j$.
	The map $h_k$ is defined precisely in such a way that the value of $h_k$ at a right endpoint of an interval in~$I$ is equal to its
	value at the successive left endpoint of an interval in~$I$. Thus we can now split the maps $h_1, \dots, h_m$ for $r = q$.
	In this way we obtain a partition $T'_1, \dots, T'_q$ of $[(q-1)m+1]$ and points $y = t'_0 \le t'_1 \le \dots \le t'_{(q-1)m+1} = z$, 
	$t'_i \in I$ for $i \in [(q-1)m]$, such that
	$$\sum_{j\in T'_1} h_k(t'_{j}) - h_k(t'_{j-1}) 
	= \sum_{j\in T'_2} h_k(t'_{j}) - h_k(t'_{j-1}) =  \dots 
	= \sum_{j\in T'_q} h_k(t'_{j}) - h_k(t'_{j-1}), \quad k = 1, \dots, m.$$
	The sum $\sum_{i=1}^q \sum_{j\in T'_i} h_k(t'_{j}) - h_k(t'_{j-1})$ is equal to $h_k(z)-h_k(y)$. By definition of $h_k$ this
	is equal to $\sum_{j\in T_i} f_k(t_j) - f_k(t_{j-1}) = \frac{1}{p}(f_k(1) -f_k(0))$. Thus $\sum_{j\in T'_i} h_k(t'_{j}) - h_k(t'_{j-1}) = \frac{1}{pq}(f_k(1) -f_k(0))$
	for all~$i$ and~$k$. Now if $t'_{j-1}$ and $t'_j$ are in the same interval $[t_{\ell-1}, t_\ell]$, then $h_k(t'_{j}) - h_k(t'_{j-1}) 
	= f_k(t'_{j}) - f_k(t'_{j-1})$. Whereas if $t'_{j-1} \in [t_{\lambda-1}, t_\lambda]$ and $t'_j \in [t_{\ell-1}, t_\ell]$, then
	\begin{align*}
		&h_k(t'_{j}) - h_k(t'_{j-1}) \\
		&= f_k(t'_j) - f_k(t_{\ell-1}) + \left(\sum_{s < \ell, s \in T_i} f_k(t_s) - f_k(t_{s-1})\right) 
		- \left[f_k(t'_{j-1}) - f_k(t_{\lambda-1}) + \left(\sum_{s < \lambda, s \in T_i} f_k(t_s) - f_k(t_{s-1})\right)\right] \\
		&= f_k(t'_j) - f_k(t_{\ell-1}) + \left(\sum_{\lambda \le s < \ell, s \in T_i} f_k(t_s) - f_k(t_{s-1})\right) + f_k(t'_{j-1})  - f_k(t_{\lambda-1}).
	\end{align*}
	Let $T''$ be the set of points $\{t_0, \dots, t_{n+1}, t'_0, \dots, t'_{(q-1)m+1}\}$, and write $t''_0 < t''_1 < \dots < t''_N$ for the 
	points in~$T''$. Let $T''_i \subset [N]$ be the set of indices corresponding to points in $T'_i$ and for any pair of consecutive 
	points in $T'_i$ add those indices corresponding to all points of $\{t_0, \dots, t_{n+1}\}$ that are between them. 
	Then by the above calculations 
	$$\frac{1}{pq}(f_k(1) -f_k(0)) = \sum_{j\in T'_i} h_k(t'_{j}) - h_k(t'_{j-1}) = \sum_{j \in T''_i} f_k(t''_{j}) - f_k(t''_{j-1}).$$
	
	The total number of points required for this division is $(p-1)m+p(q-1)m = (pq-1)m$. This completes the induction on prime divisors.
\end{proof}

For the reader who found the induction on the number of prime divisors in the proof above difficult to follow, we mention that we use
Theorem~\ref{thm:splitting-necklaces} for all integers $r \ge 2$ only to show Theorem~\ref{thm:splitting-loops} in full generality. 
But the induction on prime divisors for this latter theorem is of much lower technical difficulty. 

To prove results about inscribing parallelograms and rectangles we need a Hobby--Rice theorem for maps defined on the circle~$S^1$. 
Consider the following first approximation to the desired result: For any $m \ge 2$ continuous maps $f_1, \dots, f_m\colon S^1 \longrightarrow \R$ 
one can find $m$ points $t_1, \dots, t_m \in S^1$ and a partition $T_1 \sqcup T_2$ of $[m]$ such that 
$\sum_{j\in T_1} f_k(t_j) - f_k(t_{j-1}) = \sum_{j\in T_2} f_k(t_j) - f_k(t_{j-1})$ for all~$k$. Here $t_0$ denotes~$t_m$. As stated this result 
trivially holds for $t_1 = t_2 = \dots = t_m$. To avoid this degeneracy we will cut the circle $S^1$ open at an arbitrary point to obtain
maps $f_i \colon [0,1] \longrightarrow \R$, and we will always require that at least one map, say~$f_m$, satisfies $f_m(0) \ne f_m(1)$,
that is, $f_m$ did not come from a map defined on~$S^1$. Then the above theorem holds true if $m$ is even (and is false for odd $m$ by
a degrees of freedom counting argument). We will mostly need the following special case:

\begin{corollary}
\label{cor:4maps}
	Let $f_1, \dots, f_4\colon [0,1] \longrightarrow \R$ be continuous functions.
	Then there are points $0 \le t_1 \le \dots \le t_4 \le 1$ such that
	$$2f_k(t_1)+2f_k(t_3)+f_k(1) = 2f_k(t_2) + 2f_k(t_4) + f_k(0) \quad \text{for all} \ k.$$
\end{corollary}

\begin{proof}
	We use Theorem~\ref{thm:splitting-necklaces} with $r=2$ and $m=4$. This provides us with four points
	$0 \le t_1 \le \dots \le t_4 \le 1$ and a partition $T_1 \sqcup T_2$ of~$[5]$. If $T_1 = \{1,3,5\}$ and $T_2 = \{2,4\}$ (or vice versa)
	then the conclusion of Theorem~\ref{thm:splitting-necklaces} is equivalent to $2f_k(t_1)+2f_k(t_3)+f_k(1) = 2f_k(t_2) + 2f_k(t_4) + f_k(0)$.
	For any other partition of $[5]$ at least one of the $T_i$ has successive elements. Suppose $j$ and $j+1$ are in $T_1$ (say) 
	and they are the largest successive pair of numbers in the same~$T_i$. Swap $j+1$ into~$T_2$, $j+2$ into~$T_1$, and so on up to $j+\ell = 5$.
	Call the new partition of $[5]$ obtained in this way $T'_1 \sqcup T'_2$.
	Forget the point $t_j$ and reindex to get new points~$t'_i$ as follows: $t'_1 = t_1, \dots, t'_{j-1} = t_{j-1}$, $t'_j = t_{j+1}, \dots, t'_3 = t_4$,
	and $t'_4 = 1$. The equation $\sum_{j\in T_1} f_k(t_j) - f_k(t_{j-1}) = \sum_{j\in T_2} f_k(t_j) - f_k(t_{j-1})$ is
	equivalent to $\sum_{j\in T'_1} f_k(t'_j) - f_k(t'_{j-1}) = \sum_{j\in T'_2} f_k(t'_j) - f_k(t'_{j-1})$. So we can successively 
	reduce to the case $T_1 = \{1,3,5\}$ and $T_2 = \{2,4\}$.
\end{proof}

We can now prove Hadwiger's conjecture that any simple loop in $\R^3$ inscribes a parallelogram. In fact, any such loop inscribes
many parallelograms: their vertex sets are dense in the image of the loop. We consider four pairwise distinct points on a line to be a parallelogram
if they arise as the limit of a sequence of parallelograms, and Hadwiger~\cite{hadwiger1971} explicitly allows this.

\begin{theorem}
\label{thm:hadwiger}
	Any simple loop $\gamma \colon [0,1] \longrightarrow \R^3$ inscribes sufficiently many parallelograms that their vertex sets 
	are dense in~$\gamma([0,1])$.
\end{theorem}

\begin{proof}
	Apply Corollary~\ref{cor:4maps} to the coordinate functions $\gamma_1, \gamma_2, \gamma_3$, of~$\gamma$, and to the function
	$$f(t) = \begin{cases} 0 & \text{ if } t\in [0,x]\\  \frac{1}{y-x}(t-x) 
	& \text{ if } t \in [x,y]\\ 1 & \text{ if } t\in [y,1]\end{cases}$$
	for a given interval $[x,y] \subset [0,1]$. Let $0 \le t_1 \le \dots \le t_4 \le 1$ be the points whose existence is guaranteed
	by Corollary~\ref{cor:4maps}.
	
	Since $\gamma$ is a loop, we have that $\gamma(0) = \gamma(1)$ and thus $\gamma(t_1)+\gamma(t_3) = \gamma(t_2)+\gamma(t_4)$.
	So the points $\gamma(t_1), \dots, \gamma(t_4)$ form a (possibly degenerate) parallelogram inscribed into~$\gamma$.
	Moreover, we know that $2f(t_1)+2f(t_3)+1 = 2f(t_2) + 2f(t_4)$. This does not have a solution where all $f(t_i)$ are integers. 
	Thus at least one $t_i$ is in the interval~$(x,y)$. Since this is true for any open interval $(x,y) \subset [0,1]$, we conclude that
	the set of vertices of inscribed parallelograms is dense in~$\gamma([0,1])$. 
	
	Lastly, we check that $f$ prevents the parallelogram from being degenerate. If $t_1 = t_2$, then $\gamma(t_1)+\gamma(t_3) = \gamma(t_2)+\gamma(t_4)$
	implies that $t_3=t_4$ since $\gamma$ is an embedding, but this directly contradicts $2f(t_1)+2f(t_3)+1 = 2f(t_2) + 2f(t_4)$. The case $t_2 = t_3$
	is similar.
\end{proof}

To prove results about inscribed rectangles, we need a lemma that distinguishes rectangles among parallelograms.
The British Flag Theorem states that if $ABCD$ are the vertices of a rectangle in a plane (in cyclic order) and $P \in \R^2$
is any point then $|PA|^2+|PC|^2 = |PB|^2 + |PD|^2$. We will need the converse of the British Flag Theorem:

\begin{lemma}
\label{lem:converse-british}
	Let $A,B,C,D \in \R^2$ be the vertices of a parallelogram in counterclockwise order. If there is a point $P \in \R^2$ such
	that $|PA|^2+|PC|^2 = |PB|^2 + |PD|^2$, then $ABCD$ is a rectangle.
\end{lemma}

\begin{proof}
	Choose coordinates with the intersection of the diagonals of the parallelogram at the origin. Thus $A = -C$ and $B = -D$,
	and $|P+A|^2 + |P-A|^2 = |P+B|^2 + |P-B|^2$. This is equivalent to $2|P|^2 + 2|A|^2 = 2|P|^2 + 2|B|^2$ and thus
	$|A|^2=|B|^2=|C|^2=|D|^2$, so $ABCD$ is a rectangle.
\end{proof}

We can now prove the existence of many inscribed rectangles. Recently and independently, Schwartz~\cite{schwartz2018}
proved a trichotomy for inscribed rectangles in planar loops showing that all but at most four points are the vertices of
inscribed rectangles.

\begin{theorem}
\label{thm:many-rectangles}
	Let $\gamma\colon [0,1] \longrightarrow \R^2$ be a simple loop. Then $\gamma$ 
	inscribes sufficiently many non-degenerate rectangles that the set of vertices is dense in~$\gamma([0,1])$.
\end{theorem}

\begin{proof}
	Apply Corollary~\ref{cor:4maps} to the following functions: $\gamma_1, \gamma_2$, the function $f$
	from the proof of Theorem~\ref{thm:hadwiger}, and $g(t) = |\gamma(t)|^2$. Then the functions $\gamma_1$,
	$\gamma_2$, and $f$ guarantee that we obtain a non-degenerate inscribed parallelogram with at least
	one vertex in~$\gamma((x,y))$ for some arbitrary interval $(x,y) \subset [0,1]$. The function $g$ ensures
	that the parallelogram is actually a rectangle by Lemma~\ref{lem:converse-british}.
\end{proof}

\begin{example}
	In general we cannot prescribe a vertex of an inscribed rectangle precisely.
	Consider a curve $\gamma$ that traces a triangle. Then we cannot prescribe a vertex of an inscribed 
	rectangle to be a vertex of the triangle at an acute angle.
\end{example}

\section{Splitting rectifiable loops}

We started Section~\ref{sec:rectangles} by showing that the Hobby--Rice theorem implies that any planar $C^1$ loop
inscribes a parallelogram with one vertex at~$\gamma(0)$. We used Equation~(\ref{eq:der}) to ensure that the parallelogram
is non-degenerate. This equation more generally asserts that the length of $\gamma$ over the intervals $[0,a]$ and $[b,c]$
is equal to length over the intervals $[a,b]$ and~$[c,1]$. Thus $\gamma$ is cut into four pieces $\gamma|_{[0,a]}$,
$\gamma|_{[a,b]}$, $\gamma|_{[b,c]}$, and $\gamma|_{[c,1]}$ such that the pieces can be translated to form two loops
of equal length. In this section we extend this result to higher dimensions and splitting into more than two loops of equal
length. 

For the notion of length to be well-defined the loop $\gamma$ needs to be rectifiable.
A curve $\gamma\colon [0,1] \longrightarrow \R^d$ is called \emph{rectifiable} if there is a constant $C > 0$ such that
$$\sum_{j=1}^{n-1} {|\gamma(x_{j+1}) - \gamma(x_j)|} < C$$ for any $n$ and any set of points $x_1 < x_2 < \dots < x_n$ in~$[0,1]$.
In particular, the length of a rectifiable curve is well-defined. A rectifiable curve $\gamma\colon [0,1] \longrightarrow \R^d$ can 
be parametrized by arc length.

\begin{theorem}
\label{thm:splitting-loops}
	Let $\gamma\colon [0,1] \longrightarrow \R^d$ be a rectifiable loop.
	For an integer $r \ge 2$, let $n = (r-1)(d+1)$. Then there exists a partition of $[0,1]$ into $n+1$ intervals $I_1, \dots, I_{n+1}$ 
	by $n$ cuts and a partition of the index set $[n+1]$ into subsets $T_1, \dots, T_r$ such that the restrictions $\gamma|_{I_j}$, 
	$j \in T_k$, can be rearranged by translations to form a loop for each $k \in \{1, \dots, r\}$, and these $r$ loops all have the
	same length.
\end{theorem}

\begin{proof}	Parametrize $\gamma$ by arc length and apply Theorem~\ref{thm:splitting-necklaces} 
	to the $d$ coordinate functions $\gamma_1, \dots, \gamma_d$ and the function $f(t) = t$. Then
	$$\sum_{j\in T_1} \gamma(t_j) - \gamma(t_{j-1}) = \sum_{j\in T_2} \gamma(t_j) - \gamma(t_{j-1}) =  \dots 
	= \sum_{j\in T_r} \gamma(t_j) - \gamma(t_{j-1})$$
	implies that $\sum_{j\in T_i} \gamma(t_j) - \gamma(t_{j-1}) = 0$ for all $i \in [r]$. Thus the pieces 
	$\gamma|_{[t_{j-1},t_j]}$, $j \in T_i$, of $\gamma$ can be rearranged by translations to form a loop for each~$i \in [r]$.
	Moreover, $\sum_{j\in T_1} t_j - t_{j-1} = \sum_{j\in T_2} t_j - t_{j-1} =  \dots = \sum_{j\in T_r} t_j - t_{j-1}$ implies that
	these $r$ loops have the same length, since $\gamma$ is parametrized by arc length. 
\end{proof}	

In particular, for $r = 2$ and $d = 3$ Theorem~\ref{thm:splitting-loops} implies that any simple loop $\gamma$ in $\R^3$ inscribes
a parallelogram whose vertices cut $\gamma$ into four pieces $\gamma^{(1)}, \gamma^{(2)}, \gamma^{(3)}, \gamma^{(4)}$ in cyclic order such that
$\gamma^{(1)}$ and $\gamma^{(3)}$ have the same total length as $\gamma^{(2)}$ and~$\gamma^{(4)}$. 

Theorem~\ref{thm:many-rectangles}
asserts that any simple planar loop inscribes many rectangles. While we have been unable to use this to derive Toeplitz' conjecture
that one of these rectangles is a square, we can use similar reasoning to that used in the proof of Theorem~\ref{thm:splitting-loops}
to ensure that the length of the loop over pairs of opposite sides of the rectangle is the same. That is, instead of the sides of the
rectangle itself having the same length, we can only ensure this for the pieces of the loop over those sides.

\begin{theorem}
\label{thm:inscribed}
	Let $\gamma\colon [0,1] \longrightarrow \R^2$ be a simple rectifiable loop.
	The loop $\gamma$ inscribes a non-degenerate rectangle cutting it into four pieces $\gamma^{(1)}, \gamma^{(2)}, 
	\gamma^{(3)}, \gamma^{(4)}$ in cyclic order such that $\gamma^{(1)}$ and $\gamma^{(3)}$ have the same total length as
	$\gamma^{(2)}$ and~$\gamma^{(4)}$.
\end{theorem}

\begin{proof}
	Parametrize $\gamma$ by arc length.
	Use Corollary~\ref{cor:4maps} for $\gamma_1$, $\gamma_2$, $g(t) = |\gamma(t)|^2$, and $f(t) = t$.
	The first three functions ensure a (possibly degenerate) inscribed rectangle, while $f$ guarantees that
	the total length of $\gamma^{(1)}$ and $\gamma^{(3)}$ is equal to that of $\gamma^{(2)}$ and~$\gamma^{(4)}$.
\end{proof}

\section{Necklace splittings with additional constraints}
\label{sec:necklace-constraints}

In this section we prove a proper strengthening of Alon's necklace splitting result for $r$ a prime. This in turn yields a strengthened
loop splitting result, provided that the number of resulting loops $r$ is a prime. We find it noteworthy that for these results the usual
induction on the number of prime divisors seems to fail entirely. We are unable to derive similar results for non-primes~$r$. In fact,
a result of Blagojevi\'c, Matschke, and Ziegler~\cite{blagojevic2015} implies that the topological method used in the proof fails outside 
of the case that $r$ is a prime. In light of the recent counterexamples to the topological Tverberg conjecture for parameters that are 
not prime powers~\cite{blagojevic2018, frick2015, mabillard2015}, this opens the interesting question of whether the primality of $r$ 
is perhaps not an artifact of our proof method, but actually an essential prerequisite of our result. 

Generalizations of Theorem~\ref{thm:splitting-necklaces} of various kinds have recently received much attention; see for example 
de Longueville and \v Zivaljevi\'c~\cite{deLongueville2008}, Karasev, Rold\'an-Pensado, and Sober\'on~\cite{karasev2016}, Alishahi 
and Meunier~\cite{alishahi2017}, Asada et al.~\cite{asada2017}, and Blagojevi\'c and Sober\'on~\cite{blagojevic2017}. Here we
show the following:

\begin{theorem}
\label{thm:opt-splitting-necklaces}
	Let $f_1, \dots, f_m\colon [0,1] \longrightarrow \R$ be continuous functions. For a prime $r \ge 2$ let~${n = (r-1)m}$.
	Let $C_1, \dots, C_\ell$ be a partition of $[n+1]$ with $|C_i| \le r-1$.
	Then there are points $0 = t_0 \le t_1 \le \dots \le t_{n+1} = 1$ and a partition of the index
	set $[n+1]$ into subsets $T_1, \dots, T_r$ with $|C_i \cap T_j| \le 1$ for every $i$ and~$j$ such that
	$$\sum_{j\in T_1} f_k(t_j)-f_k(t_{j-1}) = \sum_{j\in T_2} f_k(t_j)-f_k(t_{j-1}) =  \dots = \sum_{j\in T_r} f_k(t_j)-f_k(t_{j-1}), \quad k = 1, \dots, m.$$
\end{theorem}

In the usual interpretation of Alon's result, where $[0,1]$ is thought of as an unclasped necklace with $m$ types of beads whose
density along the necklace is given by $\mu_1, \dots, \mu_m$ and the sets $T_i$ are thieves who would like to split the necklace
fairly, the result above guarantees that there are blocks of size at most $r-1$ pieces of the necklace such that no thief receives 
two pieces of the necklace within such a block.

Compare Theorem~\ref{thm:opt-splitting-necklaces} with the following optimal colored Tverberg theorem of Blagojevi\'c, Matschke, and Ziegler.

\begin{theorem}[Blagojevi\'c, Matschke, and Ziegler~\cite{blagojevic2015}]
	Let $r \ge 2$ be a prime and $d \ge 1$ be an integer. Let $n =(r-1)(d+1)$, and let $C_1, \dots, C_\ell$ be a partition of the
	vertex set of the $n$-simplex~$\Delta_n$ with~$|C_i|\le r-1$ for all~$i$. Then for any continuous map $f \colon \Delta_n \longrightarrow \R^d$ 
	there are $r$ pairwise disjoint faces $\sigma_1, \dots, \sigma_r$ of $\Delta_n$ such that each $\sigma_i$ has at most one vertex in
	each $C_j$ and with $f(\sigma_1) \cap \dots \cap f(\sigma_r) \ne \emptyset$. 
\end{theorem}

To prove Theorem~\ref{thm:opt-splitting-necklaces} we combine the central topological result of~\cite{blagojevic2015} with Matou\v sek's proof
of Theorem~\ref{thm:splitting-necklaces} and a combinatorial reduction to a special case; see Lemma~\ref{lem:special-case}.
The complex $[n]^{*m}_\Delta$ denoted $\Delta_{n,m}$ is called the chessboard complex. Here $[n]$ denotes
the $0$-dimensional simplicial complex on vertex set~$[n]$. The symmetric group $S_n$ naturally acts on~$[n]$, and the subgroup
$\Z/n$ acts by shifts. Thus these groups act diagonally on joins and deleted joins of these complexes, in particular, on chessboard
complexes~$\Delta_{n,m}$. 

We can now state the central topological lemma needed for the proof of Theorem~\ref{thm:opt-splitting-necklaces}. See also 
Vre\'cica and \v Zivaljevi\'c~\cite{vrecica2011}.

\begin{lemma}[Blagojevi\'c, Matschke, and Ziegler~\cite{blagojevic2015}]
\label{lem:equivariant-map}
	Let $r \ge 2$ be a prime, $m \ge 1$ and integer, and $n = (r-1)m$. Then any $\Z/r$-equivariant map 
	$(\Delta_{r,r-1})^{*m} * [r] \longrightarrow W_r^{\oplus m}$ must have a zero.
\end{lemma}

The following lemma is analogous to a reduction in~\cite{blagojevic2015} for Tverberg-type results.

\begin{lemma}
\label{lem:special-case}
	It is sufficient to prove Theorem~\ref{thm:opt-splitting-necklaces} in the case that $\ell = m+1$, $|C_i| = r-1$ for $i < \ell$
	and $|C_\ell| = 1$.
\end{lemma}

\begin{proof}
	We are given continuous functions $f_1, \dots, f_m\colon [0,1] \longrightarrow \R$, a prime $r \ge 2$, and~${n = (r-1)m}$.
	Let $C_1, \dots, C_\ell$ be a partition of $[n+1]$ with $|C_i| \le r-1$. Certainly $\ell$ is larger than~$m$. We define $N$
	to be the integer~${(r-1)\ell}$, and we enlarge the sets $C_i$ and add the new set $C_{\ell+1}' = \{N+1\}$ to be a partition of~$[N+1]$.
	More precisely, obtain $C'_i$ from $C_i$ by adding $r-1-|C_i|$ elements in $[N] \setminus [n+1]$; this can be done in 
	such a way that $C'_1, \dots, C_{\ell+1}'$ is a partition of~$[N+1]$.
	
	Define the functions $h_1, \dots, h_m \colon [0,1] \longrightarrow \R$ by $h_i(x) = f_i(2x)$ for $x \le \frac12$ and $h_i(x) = f_i(1)$
	for $x > \frac12$. Let $[a_1,b_1], \dots, [a_{\ell-m}, b_{\ell-m}]$ be pairwise disjoint intervals in $[\frac12, 1]$. Define $\ell-m$ new functions 
	$h_{m+1}, \dots, h_\ell \colon [0,1] \longrightarrow \R$ by $h_i(x) = 0$ for $x < a_{i-m}$, $h_i(x) = 1$ for $x > b_{i-m}$, and interpolate
	linearly in between, that is, $h_i(x) = \frac{1}{b_{i-m}-a_{i-m}}(x-a_{i-m})$ for $x \in [a_{m-i}, b_{m-i}]$.
	When we assume that Theorem~\ref{thm:opt-splitting-necklaces} has been
	shown for $|C'_i| = r-1$ for $i \le \ell$ and $|C'_{\ell+1}| = 1$, then we can find points $0 = t_0 \le t_1 \le \dots \le t_{N+1} = 1$ 
	and a partition $T_1, \dots, T_r$ of $[N+1]$ such that 
	$$\sum_{j\in T_1} h_k(t_j)-h_k(t_{j-1}) = \sum_{j\in T_2} h_k(t_j)-h_k(t_{j-1}) =  \dots = \sum_{j\in T_r} h_k(t_j)-h_k(t_{j-1}), \quad k = 1, \dots, m$$
	and $|C'_i \cap T_j| \le 1$ for each $i$ and~$j$.
	
	Of the points $t_i$ at least $r-1$ points need to be in each interval~$[a_i,b_i]$, which requires $(r-1)(\ell-m)$ points in total.
	Thus at most $(r-1)m$ points $t_i$ are contained in the interval~$[0, \frac12]$. But then 
	$$\sum_{j\in T_1} f_k(2t_j)-f_k(2t_{j-1}) = \sum_{j\in T_2} f_k(2t_j)-f_k(2t_{j-1}) =  \dots = \sum_{j\in T_r} f_k(2t_j)-f_k(2t_{j-1}), \quad k = 1, \dots, m$$
	for those points~$t_i$, proving the general case of Theorem~\ref{thm:opt-splitting-necklaces}.
\end{proof}

\begin{proof}[Proof of Theorem~\ref{thm:opt-splitting-necklaces}]
	By the reduction of Lemma~\ref{lem:special-case} we only need to consider the case that $\ell = m+1$ with $|C_1| = \dots = |C_{\ell-1}| = r-1$
	and~${|C_\ell| = 1}$, which we will do from here on. We construct the $S_r$-equivariant map
	$F\colon (\Delta_n)^{*r}_\Delta \longrightarrow (\R^m)^r$ as in the proof of Theorem~\ref{thm:splitting-necklaces}.
	Since we identified the vertex set of $\Delta_n$ with $[n+1]$ each set $C_i$ is a subset of the vertex set of the $n$-simplex,
	and thus $(C_1 * \dots * C_\ell)^{*r}_\Delta$ is an $S_r$-invariant subcomplex of~$(\Delta_n)^{*r}_\Delta$. A point
	$x \in (C_1 * \dots * C_\ell)^{*r}_\Delta$ precisely corresponds to points $0 = t_0 \le t_1 \le \dots \le t_{n+1} = 1$ and 
	a partition $T_1, \dots, T_r$ of $[n+1]$ as in the statement of the theorem. Observe that if the theorem was false,
	then the image of $F$ restricted to $(C_1 * \dots * C_\ell)^{*r}_\Delta$ would not intersect the diagonal 
	$D = \{(y_1, \dots,y_r) \in (\R^m)^r \ | \ y_1 = \dots = y_r\}$.  Orthogonally projecting along the diagonal gives
	an equivariant map $\widehat F\colon (C_1 * \dots * C_\ell)^{*r}_\Delta \longrightarrow W_r^{\oplus m}$ that
	does not map to zero. 
	
	Now since $|C_1| = \dots = |C_{\ell-1}| = r-1$ and~${|C_\ell| = 1}$ and since joins and deleted joins commute the 
	complex $(C_1 * \dots * C_\ell)^{*r}_\Delta$ is isomorphic to $([r-1]^{*r}_\Delta)^{*(\ell-1)} * [1]^{*r}_\Delta \cong 
	(\Delta_{r,r-1})^{*(\ell-1)} * [r]$. Thus $\widehat F$ contradicts Lemma~\ref{lem:equivariant-map}.
\end{proof}

The same topological machinery fails for non-primes $r$; see Blagojevi\'c, Matschke, and Ziegler~\cite{blagojevic2015}. 
In the same way that Theorem~\ref{thm:splitting-loops} follows from Theorem~\ref{thm:splitting-necklaces}, we can derive
the following corollary from Theorem~\ref{thm:opt-splitting-necklaces}.

\begin{corollary}
\label{cor:opt-splitting-loops}
	Let $\gamma\colon [0,1] \longrightarrow \R^d$ be a rectifiable loop. For a prime $r \ge 2$, let $n = (r-1)(d+1)$. 
	Let $C_1, \dots, C_m$ be a partition of~$[n+1]$ with~$|C_i| \le r-1$. Then there exists a partition of $[0,1]$ into $n+1$ 
	intervals $I_1, \dots, I_{n+1}$ by $n$ cuts and a partition of the index set $[n+1]$ into subsets $T_1, \dots, T_r$ 
	with $|C_i \cap T_k| \le 1$ such that the restrictions $\gamma|_{I_j}$, $j \in T_k$, can be rearranged by translations to form a loop for 
	each $k \in \{1, \dots, r\}$, and these $r$ loops all have the same length.
\end{corollary}

\begin{question}
	Is the condition that $r$ is a prime actually required in Theorem~\ref{thm:opt-splitting-necklaces} and Corollary~\ref{cor:opt-splitting-loops}?
\end{question}

\section*{Acknowledgements}

This research was performed during the Summer Program for Undergraduate Research 2017 at Cornell
University. The authors are grateful for the excellent research conditions provided by the program. The
authors would like to thank Maru Sarazola for many insightful conversations. The authors would also like 
to thank Camil Muscalu, Phil Sosoe, and Gennady Uraltsev for clarifying discussions.


\end{document}